\documentclass[12pt]{article}

\usepackage{hyperref}

\usepackage{amsfonts, amsmath, amssymb, amsthm}

\usepackage{a4}

\usepackage{blkarray}

 \usepackage{xcolor}

 \definecolor{Refkey}{RGB}{255,127,0}
 \definecolor{Labelkey}{RGB}{127,0,255}
 \makeatletter 
  \def\SK@refcolor{\color{Refkey}}
  \def\SK@labelcolor{\color{Labelkey}}
 \makeatother

  \definecolor{mdg}{RGB}{0,177,0} 
  \definecolor{mdb}{RGB}{0,0,191}
  \definecolor{mddb}{RGB}{0,0,91}
  \definecolor{mdy}{RGB}{255,69,0} 
  \definecolor{gray}{RGB}{99,99,99} 

\usepackage{upgreek}

\DeclareMathOperator{\const}{const}
\DeclareMathOperator{\rank}{rank}

\newtheorem{theorem}{Theorem}
\newtheorem{proposition}{Proposition}

\theoremstyle{definition}
\newtheorem{definition}{Definition}
\newtheorem{convention}{Convention}

\theoremstyle{remark}
\newtheorem*{remark}{Remark}

\title{Heptagon relations from a simplicial 3-cocycle, and their cohomology}
\author{Igor G. Korepanov}

\date{July 2022 -- May 2025}

\begin{document}

\sloppy

\maketitle

\medskip

\begin{abstract}
We introduce new algebraic structures associated with heptagon relations---higher analogue of the well-known pentagon. The main points we deal with are: (i)~polygon relations as algebraic imitations of Pachner moves, on the example of heptagon, (ii)~parameterization of heptagon relations by simplicial 3-cocycles, (iii)~applications to invariants of pairs "piecewise linear 5-manifold, a 3rd cohomology class on it".
\end{abstract}

\section{Introduction}\label{s:i}

This paper is about new algebraic structures appearing in the study of \emph{polygon relations}.

\subsection{Pentagon---the simplest polygon relation}\label{ss:p}

 Polygon relations appear as a generalization of the popular \emph{pentagon} relations (also called ``solutions to pentagon equation''). Pentagon relations are often written as
\begin{equation}\label{pentagon-constant}
 R_{12} R_{13} R_{23} = R_{23} R_{12}.
 \end{equation}
The exact meaning of~\eqref{pentagon-constant} may, however, be quite different.

The most popular version of~\eqref{pentagon-constant} seems to be \emph{constant set-theoretic}. In it, there are three copies $X_1$, $X_2$, $X_3$ of a set~$X$, and $R_{ij}$ maps the Cartesian product $X_i \times X_j$ into itself, all~$R_{ij}$ also being copies of each other. The whole relation~\eqref{pentagon-constant} takes place in $X_1 \times X_2 \times X_3$, see for instance~\cite[Section~1]{pcst}.

There exists also a \emph{nonconstant} version of~\eqref{pentagon-constant}:
\begin{equation}\label{pentagon-nonconstant}
 R_{12}^{(1)} R_{13}^{(2)} R_{23}^{(3)} = R_{23}^{(4)} R_{12}^{(5)},
\end{equation}
where all $R$'s can be different and, moreover, the sets~$X_i$ in whose direct products they act may also be all different. An interesting particular case of such relations is where sets~$X_i$ are \emph{linear spaces} over a field, and mappings~$R_{ij}^{(a)}$ are linear, hence, can be represented by matrices in the finite-dimen\-sional case. Nontrivial relations of this kind can be found in~\cite{KS}, and we will mention them below in Section~\ref{s:po} while discussing the pentagon analogue of a construction proposed in this paper.

One more realization of \eqref{pentagon-constant} or~\eqref{pentagon-nonconstant} that must be mentioned is \emph{quantum relations}. In that case, each $X_i$ is a linear space, while each $R_{ij}^{(a)}$ is a linear operator acting in the \emph{tensor product} $X_i \otimes X_j$; the whole relation takes place in $X_1 \otimes X_2 \otimes X_3$. The interest in quantum relations originates from the Yang--Baxter equation that appeared originally in its quantum form; Yang--Baxter is the first in the hierarchy of \emph{simplex equations/relations}, and these are often cited together with polygon relations, due to some parallelism and, mainly, nontrivial interrelations between the two hierarchies, see for instance~\cite{DM-H} and~\cite{YB-from-pentagon}.

\subsection{Polygon relations as algebraic imitations of Pachner moves}\label{ss:im}

Pentagon relations are known to correspond, geometrically, to \emph{Pachner move 2--3}. This means the replacement of two adjacent tetrahedra (forming the \emph{star} of their common 2-face) in a triangulation of a 3-manifold with three tetrahedra forming the star of an edge introduced instead of the mentioned 2-face.

Concerning polygon relations---higher analogues of pentagon---we mean by that name algebraic equalities that can be said to imitate higher-dimen\-sional Pachner moves. More specifically, $(d+2)$-gon relations are associated with $d$-dimen\-sional moves. The most promising kinds of these relations can, however, no longer (for $d>3$) be written as the equalness of two compositions of mappings, like \eqref{pentagon-constant} or~\eqref{pentagon-nonconstant}. Instead, we introduce a more sophisticated kind of composition of what we call \emph{permitted colorings} of $d$-simplices, on the example of heptagon ($d=5$).

Earlier, similar constructions have been shown to work for \emph{hexagon} ($d=4$)~\cite{nonconstant-hexagon}, and there is no problem to interpret this way any set-theoretic pentagon relation as well.

Note that there exists also a quite different definition of polygon relations in the literature, leading (if $d>3$) to different objects~\cite{DM-H}.

\subsection{Parameterization by simplicial cocycles}\label{ss:pc}

One appealing feature of our constructions is that our polygon relations are naturally parameterized by \emph{simplicial cocycles}---3-cocycles in the heptagon case, 2-cocycles for hexagon~\cite{nonconstant-hexagon}, and actually 1-cocycles for pentagon, see~\cite{KS} and Section~\ref{s:po} below.

\subsection{Topological applications}\label{ss:ta}

This paper is about some new \emph{algebra}; no achievements are claimed, at this moment, in the topology of five-dimen\-sional manifolds. Still, we find it reasonable to point at possible applications to \emph{invariants of pairs} "piecewise linear 5-manifold, a 3rd cohomology class on it". We provide a few relatively simple calculation examples, for some not very complicated 5-manifolds, that give a \emph{hope} for really interesting topological applications in the future.

\subsection{Contents of the rest of the paper}\label{ss:r}

Below,
\begin{itemize}\itemsep 0pt
 \item in Section~\ref{s:f}, we give general definitions concerning Pachner moves and polygon relations formulated in terms of permitted colorings of simplices, in the form convenient for our purposes,
 \item in Section~\ref{s:ac}, we introduce specific permitted colorings built out of a simplicial 3-cocycle, and show that they lead indeed to heptagon relations,
 \item in Section~\ref{s:po}, we investigate, in order to better understand the nature of the obtained heptagon relations, their simpler analogue---pentagon relations based on the same ideas. We show that these are nontrivial and already well-known,
 \item in Section~\ref{s:co}, we introduce heptagon cochains and the coboundary operator. Incorporating also the components of the mentioned 3-cocycle as parameters, we arrive at the definition of ``heptagon parametric cochain complex'',
 \item in Section~\ref{s:gi}, we explain some general ideas of how heptagon relations and heptagon 5-cocycles can lead to manifold invariants,
 \item in Section~\ref{s:b4}, we introduce a remarkable heptagon 4-cocycle, depending bilinearly on \emph{two} permitted colorings of a 4-simplex, and
 \item in Section~\ref{s:b5}, we describe a procedure that produces what we call bipolynomial 5-cocycles from this 4-cocycle,
 \item in Section~\ref{s:sb}, we introduce an invariant of a pair ``piecewise linear 5-manifold, cohomology class of a 3-cocycle on it'', in the form that we call `stable bipolynomial',
 \item in Section~\ref{s:c}, we present some calculation results for specific manifolds, using fields of characteristic two. In all examples, our bipolynomial invariant turns out to have a remarkable form, namely of a symmetric \emph{bilinear} form of the \emph{squares} of its variables.
\end{itemize}

\section{Pachner moves and polygon relations}\label{s:f}

We give here general definitions belonging to Pachner moves and polygon relations in any manifold dimension~$d$.

\subsection{Pachner moves}\label{ss:P}

Any triangulation of a closed piecewise linear (PL) manifold~$M$ can be transformed into any other triangulation by a finite sequence of \emph{Pachner moves}~\cite{Pachner,Lickorish}.

In $d$ dimensions, there are $d+1$ kinds of Pachner moves, denoted often as ``$m$--$n$'', where $m=1,\ldots,d+1$, and $n=d+2-m$. Notation $m$--$n$ means that we take a part of a triangulation consisting of $m$ simplices of the highest dimension~$d$ and forming a \emph{star} of a $(d-m+1)$-simplex, and replace it with a star of a $(d-n+1)$-simplex, consisting of $n$ simplices. Moreover, if glued together along their boundaries (that are isomorphic as simplicial complexes), these two stars must form together the boundary~$\partial \Delta^{d+1}$ of a simplex of the next dimension~$d+1$.

To be exact, the gluing goes as follows. If a $(d-1)$-face~$\Delta_1^{d-1}$ of the first star was glued to the same $(d-1)$-face in the rest of~$M$ to which we glue a~$\Delta_2^{d-1}$ of the second star, then we glue $\Delta_1^{d-1}$ to~$\Delta_2^{d-1}$. We can also take the liberty to identify each $\Delta_1^{d-1}$ with its corresponding~$\Delta_2^{d-1}$ and say that the two stars have the \emph{same boundary}.

We call the replaced star of $m$ simplices the \emph{initial cluster}, or \emph{left-hand side} (lhs) of the Pachner move, and denote it~$C_{\mathrm{ini}}$. The replacing star of $n$ simplices will be called \emph{final cluster}, or \emph{right-hand side} (rhs) of the move, and denoted~$C_{\mathrm{fin}}$.

\subsection{Permitted colorings}\label{ss:pl}

A $d$-simplex has $d+1$ faces--- $(d-1)$-cells~$\Delta^{d-1}$. Suppose we have a set~$X$, called set of \emph{colors}, with which we can color these faces. All possible colorings of a $d$-simplex---we denote it $\Delta^d$ or often simply~$v$ ---belong thus to~$X^{\times (d+1)}$ ---the $(d+1)$-th Cartesian degree of~$X$. We postulate, however, that not all colorings are \emph{permitted}, but there is a given subset $R_v \subset X^{\times (d+1)}$ of permitted colorings. If we have a simplicial complex~$K$, then for each of its $d$-simplices~$v$, its own~$R_v$ must be given.
   
We can thus call our permitted colorings `nonconstant', emphasizing the fact that the sets~$R_v$ may be different for different $d$-simplices~$v$. 

Accordingly, the corresponding polygon relations can be called `nonconstant polygon'.

\begin{definition}\label{d:pK}
We say that a \emph{permitted coloring} of a simplicial complex~$K$ is given if an element of a given color set~$X$ is assigned to each of its $(d-1)$-faces, and in such way that the restriction of this coloring onto any $d$-simplex~$v$ in~$K$ is permitted---belongs to the given subset~$R_v$.
\end{definition}

When we glue $d$-simplices together to obtain a triangulated $d$-manifold~$K=M$ with boundary~$\partial M$, the condition that $M$ must be colored permittedly results in the fact that not all colorings may appear on~$\partial M$, but only some subset $R_{\partial M} \subset X^{\times N_{d-1}}$, where $N_{d-1}$ is the number of $(d-1)$-cells in~$\partial M$.

\begin{definition}\label{d:pb}
    The restrictions of permitted colorings of~$M$ onto~$\partial M$ are called \emph{permitted colorings of~$\partial M$}.
\end{definition}

\subsection{Polygon relations and full polygon}\label{ss:fp}

\begin{definition}\label{d:hr}
We say that \emph{polygon relation} corresponding to a given Pachner move holds if the two sets~$R_{\partial M}$ of permitted boundary colorings coincide for $M$ being either the lhs or rhs of that move.
\end{definition}

As we already indicated, $\partial M$ is the same for both sides.

\medskip

Suppose first that there are given sets $X_0$, \ldots, $X_{d+1}$ of permitted colorings for all $d$-simplices $\Delta _0^d$, \ldots, $\Delta _{d+1}^d$ in the boundary of~$\Delta^{d+1}$. We can interpret $m$ of these simplices as the initial cluster~$C_{\mathrm{ini}}$ of a Pachner move, and the remaining $n$ of them as the final cluster~$C_{\mathrm{fin}}$. Here of course $m+n=d+2$, \ $m,n\ge 1$.

\begin{definition}\label{d:fc}
If the set of permitted boundary colorings of~$C_{\mathrm{ini}}$ coincides with that of~$C_{\mathrm{fin}}$ for any such $C_{\mathrm{ini}}$ and $C_{\mathrm{fin}}$, we say that $X_0$, \ldots, $X_{d+1}$ are \emph{fully compatible}.
\end{definition}

Let now there be a \emph{fixed} set~$\mathcal X$ whose elements are $(d+2)$-tuples $\{ X_0$, \ldots, $X_{d+1} \}$ of sets of permitted colorings for $\Delta _0^d$, \ldots, $\Delta _{d+1}^d$. 

\begin{definition}\label{d:fh}
If any such tuple is fully compatible, we say that $\mathcal X$ satisfies the \emph{full polygon}.
\end{definition}

\smallskip

We think of these tuples as depending on some \emph{parameters}: to each set of parameters corresponds a tuple $ \in \mathcal X$. Below, we will introduce a simplicial 3-cocycle $\omega$ whose components will serve as such parameters.

\subsection{Pachner moves in the heptagon case}\label{ss:7p}

Heptagon case appears when the dimension $d=5$. Hence, there are six types of Pachner moves: 1--6, 2--5, 3--4 and their inverses.

\section{Our actual permitted colorings and heptagon relations}\label{s:ac}

\subsection{Coloring with simplicial 3-cocycles modulo fixed cocycle~$\omega$}\label{ss:cm}
 
 Let $\omega$ be a simplicial 3-cocycle taking values in a field~$F$ and given on a simplex~$\Delta^m$. As long as we stay within just one simplex, $\omega$ is also a \emph{coboundary}:
\begin{equation}\label{odb}
\omega = \delta \beta ,
\end{equation}
or, in components,
\begin{equation}\label{odb'}
\omega _{ijkl} = \beta _{jkl}-\beta _{ikl}+\beta _{ijl}-\beta _{ijk}.
\end{equation}

\begin{convention}\label{c:i}
Here and below, notation of the form $i_0 i_1 \ldots i_n$ means an $n$-simplex with vertices denoted by natural numbers $i_0 < i_1 < \ldots < i_n$.
\end{convention}

We impose the following technical condition on~$\omega$:
\begin{equation}\label{onv}
\omega _{ijkl} \ne 0 \quad \text{for all tetrahedra }\, ijkl.
\end{equation}

  \paragraph{Coloring a 4-simplex}
  The colors of one 4-face~$\Delta^4 \subset \Delta^m$ belong, by definition, to the \emph{three}-dimen\-sional $F$-linear space called~$V_{\Delta^4}$ consisting of 3-cocycles~$\nu$ on~$\Delta^4$ taken \emph{to within adding a multiple of~$\omega$} (restricted to~$\Delta^4$). If we choose a basis in each~$V_{\Delta^4}$, we can think of these colors as belonging to~$F^3$.
  
  \paragraph{Coloring a 5-simplex}
  For a 5-simplex~$\Delta^5$, we define permitted colorings, again, as colorings given by 3-cocycles~$\nu$ considered up to adding a multiple of~$\omega$.
  
  \paragraph{Coloring higher simplices or even general simplicial complexes}
   This goes according to our already given general Definition~\ref{d:pK}. Note that this means that each 5-(sub)simplex has, generally, a~$\nu$ \emph{of its own}.

  \subsection{Full heptagon for our colorings}

\begin{proposition}\label{p:f}
The permitted colorings of the boundary of either lhs or rhs of a Pachner move correspond bijectively to the simplicial 3-cocycles\/~$\nu$ given on this boundary, taken to within adding multiples of~$\omega$.
\end{proposition}

\begin{proof}

We must show that the mapping (3-cocycles~$\nu$ modulo~$\omega$ on the boundary) $\to$ (permitted boundary colorings) is both injective and surjective. Note that there is no problem to extend a 3-cocycle from the boundary of such cluster onto the whole cluster, so the colorings induced by cocycles~$\nu$ are permitted.

\emph{Injectivity} is clear: if $\nu$ modulo~$\omega$ does not vanish everywhere on the boundary, then there is certainly a nonzero color on some~$\Delta^4$.

\emph{Surjectivity} is proved inductively, by assembling our cluster from 5-simplices, gluing one new simplex at each step---doing an \emph{inverse shelling} in the terminology of Pachner~\cite{Pachner}, starting from a single simplex~$\Delta^5$. The boundary is a PL sphere at each step. 

Suppose we have a permitted coloring of a cluster~$C$. The inductive assumption states that the restriction of this coloring onto~$\partial C$ corresponds to some $\nu$ modulo~$\omega$. We are now gluing a new~$\Delta^5$ to~$C$, with the understanding that $\omega$ is already given on this~$\Delta^5$ also.

We choose some specific~$\nu$ (not modulo~$\omega$) on~$\partial C$; there is also some cocycle~$\nu'$ on~$\Delta^5$ that gives its permitted coloring. Consider first the case where $\partial C$ and the new~$\Delta^5$ have only one 4-face~$\Delta^4$ in common. Note that the coloring of this~$\Delta^4$ is just the same as $\nu$ modulo~$\omega$ on this~$\Delta^4$, that is, $\nu$ and~$\nu'$ differ by a multiple of~$\omega$ on~$\Delta^4$, and no problem to add this multiple to~$\nu'$ so that it will become an extension of~$\nu$ onto the whole $C\cup \Delta^5$.

 Let now there be several 4-faces~$\Delta^4$ by which $\Delta^5$ is glued to~$C$. We take then one of these, $\Delta _1^4$, and glue, for a moment, our $\Delta^5$ to~$C$ along only~$\Delta _1^4$, leaving other 4-faces unglued. Then, we extend $\nu$ onto~$\Delta^5$ as above, and call this extension~$\nu _1$.

Similarly, we can take another 4-face, $\Delta _2^4$, extend $\nu$ onto~$\Delta^5$ through it, and call this extension~$\nu _2$.

As $\nu _1$ and~$\nu _2$ give the same coloring of~$\Delta^5$, they must coincide there modulo~$\omega$:
\begin{equation*}
\nu _1 - \nu _2 = \const \cdot \, \omega.
\end{equation*}
But $\Delta _1^4$ and~$\Delta _2^4$ have a common 3-face~$\Delta^3$ (as soon as they belong to the same~$\Delta^5$), belonging also to~$\partial C$, so
\begin{equation*}
\const \cdot \, \omega |_{\Delta^3} = 0,
\end{equation*}
and as $\omega |_{\Delta^3} \ne 0$ according to our condition~\eqref{onv}, we have $\const = 0$, hence $\nu _1=\nu _2$, hence $\nu$ can be extended onto the newly glued simplex without contradictions.
\end{proof}

\begin{proposition}\label{p:fh}
The full heptagon does hold for the colorings of 5-simplices described above.
\end{proposition}

\begin{proof}
This follows immediately from Definitions \ref{d:hr} (of the polygon relation corresponding to a Pachner move) and~\ref{d:fh} (of the full polygon), and Proposition~\ref{p:f}.
\end{proof}
 
\subsection{Double colorings}\label{ss:dl}
It turns out that \emph{double} colorings will be of great use for us: 
\begin{equation*}
r = ( r_1, r_2),
\end{equation*}
where $r_1$ and $r_2$ are colorings of the kind introduced above.

As long as we consider just a heptagon relation for such~$r$, it certainly holds because it means just a pair of separate relations for $r_1$ and~$r_2$. The meaning of double colorings will be clear soon when we introduce \emph{heptagon cochains} as functions depending on colorings of the simplices of a given dimension.

\section{A pentagon version of our construction gives known pentagon solutions}\label{s:po}

For a better understanding of heptagon relations introduced in Section~\ref{s:ac}, it may be instructive to take a look at their analogues in the simpler case of pentagon relations. In this case, \emph{triangles}---faces of 3-simplices, that is, tetrahedra---must be colored; this is done using 1-cocycles~$\nu$ modulo a fixed 1-cocycle~$\omega$. The latter can be written locally as $\omega _{ij} = z_i-z_j$.
 
To reveal the nature of the resulting pentagon relations, it turns out quite enough to consider just one tetrahedron~$1234$. Within each triangle~$ijk$, the linear space of 1-cocycles~$\nu$ taken modulo~$\omega$ is one-dimensional. We choose the following basis vector in it:
\begin{equation}\label{tb}
   \begin{pmatrix} 1 & 1 & 0 \end{pmatrix} \mod \omega .
\end{equation}
Here the components correspond to 1-faces $ij$, $ik$, $jk$, in this order; remember that $i<j<k$, according to Convention~\ref{c:i}.

For the tetrahedron~$1234$, the linear space of permitted colorings is two-dimen\-sional, and we choose the following basis in it:
\begin{align*}
 \nu _1 = & \begin{pmatrix} 1 & 1 & 1 & 0 & 0 & 0 \end{pmatrix} \mod \omega, \\
 \nu _2 = & \begin{pmatrix} -1 & 0 & 0 & 1 & 1 & 0 \end{pmatrix} \mod \omega .
\end{align*}
Components correspond to 1-faces $12$, $13$, $14$, $23$, $24$ and~$34$, in this order.

  \paragraph{Components of cocycles in the above bases}
 The components of $\nu _1$ and $\nu _2$ corresponding to triangles $123$, $124$, $134$, $234$ in bases~\eqref{tb} are:
\begin{equation}\label{te}
       \begin{pmatrix} \text{components of} & \nu _1 \\
                       \text{components of} & \nu _2 \end{pmatrix} =
       \begin{pmatrix} 1 & 1 & 1 & 0 \\
                       -\frac{z_1-z_3}{z_2-z_3} & -\frac{z_1-z_4}{z_2-z_4}  & 0 & 1 \end{pmatrix}
\end{equation}

The colors~$x_{ijk}$ of all triangles~$ijk$---faces of~$1234$---belong to the two-dimensional space generated by the rows of~\eqref{te}, hence two of them can be expressed in terms of the others:
\begin{equation}\label{pe}
 \begin{pmatrix} x_{124} \\ x_{234} \end{pmatrix} =
 \begin{pmatrix}\frac{(z_3-z_2)(z_4-z_1)}{(z_3-z_1)(z_4-z_2)} & \frac{(z_2-z_1)(z_4-z_3)}{(z_3-z_1)(z_4-z_2)} \\
                -\frac{z_3-z_2}{z_3-z_1} & \frac{z_3-z_2}{z_3-z_1}  \end{pmatrix}
 \begin{pmatrix} x_{123} \\ x_{134} \end{pmatrix} 
\end{equation}
 
 \paragraph{Normalization}
To link these calculations to those in~\cite{KS}, introduce \emph{normalized colors}
\begin{equation*}
 y_{ijk} \stackrel{\mathrm{def}}{=} r_{ijk} x_{ijk}, \text{ \ where \ } r_{ijk} \stackrel{\mathrm{def}}{=} \sqrt{ \frac{z_k-z_j}{(z_k-z_i)(z_j-z_i)} } .
\end{equation*}

\paragraph{Result: known orthogonal matrix for pentagon}

Then
\begin{equation*}
\begin{pmatrix} y_{124} \\ y_{234} \end{pmatrix} = \begin{pmatrix} a & b \\ c & d \end{pmatrix} \begin{pmatrix} y_{123} \\ y_{134} \end{pmatrix} ,
\end{equation*}
where
\begin{align*}
 a = d =  & \sqrt{ \frac{(z_3-z_2)(z_4-z_1)}{(z_3-z_1)(z_4-z_2)} } \\
 b = -c = & \sqrt{ \frac{(z_2-z_1)(z_4-z_3)}{(z_3-z_1)(z_4-z_2)} }
\end{align*}
One can see that this has the same structure as well-known orthogonal pentagon matrices in~\cite[Eq.~(10)]{KS}.

  \begin{remark}
   A pedantic reader can analyze this way the whole pentagon relation and see that the result is nothing else but the relation with five orthogonal matrices in~\cite[Section~3]{KS}. We think, however, that the above calculations are quite enough to see this.
  \end{remark}

 \section{Heptagon cochains and their coboundaries}\label{s:co}
 
 \subsection{Coboundary operator for polygon cochains}
 
 Consider first the situation where the set of permitted colorings is given for a simplex~$\Delta^{m+1}$. 
 \begin{definition}\label{d:cc}
 \emph{Polygon $(m+1)$-chains} (with coefficients in~$\mathbb Z$) are defined as the free group generated by these colorings. \emph{Polygon $(m+1)$-cochains} with coefficients in an abelian group~$G$ are defined as homomorphisms from the group of chains into~$G$.
 \end{definition}
 
 \begin{definition}\label{d:cb}
  The \emph{coboundary}
\begin{equation}\label{gcob}
\delta \mathfrak c_m \colon\;\;\mathcal C_{m+1} \to G
\end{equation}
of $m$-cochain~$\mathfrak c_m$ is the following mapping. For a permitted coloring~$r$ of a simplex $\Delta^{m+1} = i_0\dots i_m$,
\begin{equation}\label{cb}
(\delta \mathfrak c_m) (r) = \sum_{k=0}^{m+1} (-1)^k\, \mathfrak c_m \bigl( f_k^{(m+1)}(r) \bigr) ,
\end{equation}
where $f_k^{(m+1)}(r)$ means the restriction of coloring~$r$ onto the face $i_0\dots \widehat{i_k} \dots i_m, \quad\; 0\le k\le m$.
 \end{definition}

\subsection{Parameters $\omega$ and heptagon cochain complex}

Our actual permitted colorings for heptagon (co)chains depend on the 3-cocycle $\omega$ on the given simplex.
\begin{definition}
\emph{Parametric $m$-cochain} is a function mapping each simplicial 3-cocycle $\omega$ given on $\Delta^m$ and satisfying \eqref{onv} (that is, $\omega_{ijkl}\ne 0$) into a heptagon $m$-cochain with the set of permitted colorings corresponding to~$\omega$.
\end{definition}

Parametric $m$-cochains form an abelian group with the addition defined as the addition of heptagon $m$-cochains for each fixed~$\omega$.

Taking the \emph{same} parametric $m$-cochain for all faces of $\Delta^{m+1}$, we obtain from \eqref{cb} a homomorphism $C_{\mathrm{par}}^m \to C_{\mathrm{par}}^{m+1}$, where $C_{\mathrm{par}}^m$ is the set of all parametric $m$-cochains. This yields a \emph{heptagon parametric cochain complex}
\begin{equation}\label{pcc-hepta}
0 \to C_{\mathrm{par}}^{4} \stackrel{\delta}{\to} C_{\mathrm{par}}^{5} \stackrel{\delta}{\to} C_{\mathrm{par}}^{6} \stackrel{\delta}{\to} \dots\, .
\end{equation}

\begin{definition}\label{d:pa}
 Parametric cocycles and coboundaries are defined in the usual way as kernels and images of the corresponding arrows in~\eqref{pcc-hepta}.
\end{definition}

In this work we, however, do not deal with $m>5$.

 \section{Some generalities concerning manifold invariants from heptagon and heptagon 5-cocycles}\label{s:gi}
 
 \subsection{5-cocycle and a quantity invariant under a Pachner move}

   Let $M$ be a closed triangulated PL five-manifold, either oriented, or all our cochains take values in a field of characteristic two. Suppose we want to obtain its invariant in the form $\mathfrak c[M]$, with $\mathfrak c$ being a 5-cochain.
 
  \paragraph{Doing a Pachner move}
  We glue a $\Delta^6$ to the manifold~$M$ by some of its faces~$\Delta^5$ constituting together the lhs of a Pachner move. Denote~$M'$ the new manifold obtained by replacing this lhs with the rhs.
  
   \begin{proposition}
    $\mathfrak c[M] = \mathfrak c[M']$ if $\mathfrak c$ can be extended onto~$\Delta^6$ in such way that $\mathfrak c$ is a (simplicial) cocycle.
   \end{proposition}
   
   \begin{proof}
    The lhs and rhs of Pachner move make equal contributions due to $\mathfrak c$ being a cocycle, and taking orientations into account.
   \end{proof}
 
 \subsection{Parametric heptagon 5-cocycle and simplicial 5-cocycle}
  Here our simplicial 3-cocycle~$\omega$ comes into play. Let it be given on the whole~$M$, and note that there is no problem in extending $\omega$ onto~$\Delta^6$, even with our technical condition~\eqref{onv} (if the field of coefficients is large enough).
  
Let the value of cochain $\mathfrak c$ on each $\Delta^5$ be the same as the value of a \emph{fixed} parametric heptagon 5-cocycle, corresponding to the restriction of $\omega$ on this $\Delta^5$. We see that we obtain this way the desired simplicial 5-cocycle from parametric heptagon 5-cocycle and given $\omega$.

 \section{Bilinear heptagon 4-cocycle}\label{s:b4}

 There is a remarkable symmetric bilinear form depending on two (simplicial) 3-cocycles $\nu$ and $\eta$ on a simplex~$\Delta ^4$. We introduce first 3-cochain~$\mu$ as follows:
\begin{equation}\label{mu}
\mu _{ijkl} \stackrel{\mathrm{def}}{=} \frac{\nu _{ijkl}\eta _{ijkl}}{\omega _{ijkl}},
\end{equation}
and then define the bilinear form as the value on~$\Delta^4 = ijklm$ of its simplicial \emph{coboundary}:
\begin{equation}\label{bf4}
Q_{\Delta^4}(\nu,\eta) \stackrel{\mathrm{def}}{=} (\delta \mu)( \Delta^4 ) = \frac{\nu _{jklm}\,\eta _{jklm}}{\omega _{jklm}} - \dots + \frac{\nu _{ijkl}\,\eta _{ijkl}}{\omega _{ijkl}}
\end{equation}

Introduce permitted colorings $r_1$ and $r_2$ of $\Delta ^4$ as the equivalence classes of $\nu$ and~$\eta$, respectively, modulo~$\omega$.

\begin{proposition}\label{p:Q} 
$Q_{\Delta^4}(\nu,\eta)$ depends actually only on  $r_1$ and $r_2$.
\end{proposition}
           
\begin{proof}
This follows from the fact that $Q_{\Delta^4}(\nu,\eta)$ clearly vanishes if either $\nu$ or~$\eta$ is proportional to~$\omega$.
\end{proof}
  
  \begin{proposition}
  The heptagon 4-cochain whose value is defined on a double coloring $r = (r_1, r_2)$ of each simplex $\Delta^4$ as $Q_{\Delta^4}(\nu,\eta)$, where $\nu$ and $\eta$ represent $r_1$ and $r_2$, is a cocycle.
  \end{proposition}
  
  We will take the liberty of denoting this cocycle by the same letter~$Q$.
  
  \begin{proof}
   Considering the values of the coboundary $\delta Q$ on a given $\Delta^5$, we first note that there is no problem to find a single $\nu$ and a single $\eta$ for the whole~$\Delta^5$ (compare the proof of Proposition~\ref{p:f}). Then, the desired statement follows from the coboundary form~\eqref{bf4} of~$Q_{\Delta^4}(\nu,\eta)$.
  \end{proof}

 \section{Bipolynomial 5-cocycles in finite characteristics from the bilinear 4-cocycle in characteristic zero}\label{s:b5}
 
 \subsection{Bilinear and bipolynomial heptagon cocycles}
   As the value of our cocycle $Q$ for each $\Delta^4$ depends bilinearly on $r_1$ and~$r_2$ comprising together a double coloring of that~$\Delta^4$, we call $Q$ \emph{bilinear cocycle}.
   
   Similarly, we will introduce soon heptagon cocycles with values depending \emph{polynomially} on each of $r_1$ and~$r_2$ (and this time they will be 5-cocycles). Naturally, we will call them \emph{bipolynomial} cocycles.
   
   \smallskip
   
   We continue to use notations $\nu$ and $\eta$ for 3-cocycles representing permitted colorings $r_1$ and~$r_2$ of a given simplex (and forming together its double permitted coloring).

 \subsection{From bilinear 4-cocycle in characteristic 0 to bipolynomial 5-cocycle in characteristic~$p$}
 
 We take now cocycle $Q$ in characteristic zero. Specifically, let all components of all $\nu$, $\eta$, and~$\omega$ be \emph{indeterminates} over field~$\mathbb Q$.
 
 \smallskip
 
Let $p$ be a prime number, $k$ a natural number, and consider the following four consecutive operations on cocycle~$Q$.
\begin{enumerate}\itemsep 0pt
 \item\label{i:r} Raise each expression~$Q_{\Delta^4}(\nu,\eta)$ to the power~$p^k$.
 \item\label{i:b} Take the coboundary of the resulting cochain with components~$\bigl( Q_{\Delta^4}(\nu,\eta) \bigr)^{p^k}$.
 \item\label{i:d} Divide the result by~$p$.
 \item\label{i:e} Reduce the result modulo~$p$. In simple words: for a ratio $P/Q$ of polynomials with $P$ and $Q$ having integer coefficients and no common factors, reduce all coefficients of both $P$ and $Q$ modulo~$p$, then treat all indeterminates as indeterminates over the simple field~$\mathbb F_p$.
\end{enumerate}

More rigorously, item~\ref{i:e} means the following. For a given~$\Delta^5$, denote $\beta$, $\beta _{\nu}$ and~$\beta _{\eta}$ the respective maximal sets of linearly independent components of (the restrictions on~$\Delta^5$ of) cocycles $\omega$, $\nu$ and~$\eta$. Denote $\mathcal F = \mathbb Q(\beta)$ ---the field of rational functions of components of~$\omega$

Introduce a \emph{discrete valuation}~\cite{discr-val} on the \emph{polynomial ring}~$\mathcal F[\upbeta_{\nu},\upbeta_{\eta}]$ as follows. Let $P(\upbeta,\upbeta_{\nu},\upbeta_{\eta})$ and~$Q(\upbeta)$ be polynomials with \emph{integer coefficients} and such that there is at least one coefficient not divisible by~$p$ in both $P$ and~$Q$, then the valuation is
\begin{equation}\label{pl}
p^l \, \frac{P(\upbeta,\upbeta_{\nu},\upbeta_{\eta})}{Q(\upbeta)}\, \mapsto\, l.
\end{equation}

Reduction modulo~$p$ of the l.h.s.\ of~\eqref{pl} can be done if $l\ge 0$, and gives zero for $l>0$ and
\begin{equation*}
\frac{P(\upbeta,\upbeta_{\nu},\upbeta_{\eta})}{Q(\upbeta)}
\end{equation*}
for $l=0$. Here and below we take the liberty of denoting the polynomials and indeterminates over the prime field~$\mathbb F_p$ by the same letters as over~$\mathbb Q$.

\begin{proposition}\label{p:5c}
The above four steps \ref{i:r}--\ref{i:e} lead correctly from bilinear cocycle~$Q$ to a bipolynomial 5-cocycle of degrees $p^k$ in both permitted colorings induced by $\upbeta_{\nu}$ and~$\upbeta_{\eta}$ over the field~$\mathcal F_p = \mathbb F_p (\upbeta)$. 
\end{proposition}

\begin{proof}
We need to prove the following three points.

\textit{Feasibility of item~\ref{i:d}.}
Let $w$ be an oriented 5-simplex, and $\epsilon _v^{(w)}$ the incidence number between~$w$ and its face~$v$. As $Q$ is a cocycle,
\begin{equation}\label{cvxvyv}
\sum _{v\subset w} \epsilon _v^{(w)} Q_v(\nu, \eta) = 0.
\end{equation}
Our discrete valuation for all summands in~\eqref{cvxvyv} is zero (taking into account~\eqref{bf4}). Using~\eqref{cvxvyv}, we can express one of the summands through the rest of them and substitute in the coboundary of~$\bigl( Q_v(\nu, \eta) \bigr)^{p^k}$, that is, in
\begin{equation}\label{cxyk}
\sum _{v\subset w} \epsilon _v^{(w)} \bigl( Q_v(\nu, \eta) \bigr)^{p^k}.
\end{equation}
When we expand the result, its coefficients will clearly be all divisible by~$p$.

\textit{The result of~\ref{i:e} is a cocycle.}
This follows from the fact that the result of~\ref{i:d} is a coboundary and hence a cocycle.

\textit{The result of~\ref{i:e} depends indeed only on the pair of \emph{colorings}, determined by our cocycles $\nu$ and~$\eta$ (taken modulo~$p$ of course).}
This follows easily from the construction.
\end{proof}

\subsection{Explicit expression in characteristic two} 
A handy explicit expression for the resulting 5-cocycle can be obtained using \emph{Newton's identities}~\cite[Section~I.2]{Macdonald} for power sums and elementary symmetric functions, with the summands $\epsilon _v^{(w)} Q_v(\nu, \eta)$ in the l.h.s.\ of~\eqref{cvxvyv} taken as variables. Equality~\eqref{cvxvyv} means of course that their \emph{first power sum vanishes}.

In characteristic two, the following symbol is useful in these calculations and writing out the results:
\begin{equation*}
\tilde{\epsilon}_v^{(w)} \,\stackrel{\mathrm{def}}{=}\, \frac{\epsilon _v^{(w)}+1}{2} = 
\begin{cases} 1 & \text{if \ } \epsilon _v^{(w)}=1 \\
              0 & \text{if \ } \epsilon _v^{(w)}=-1 \end{cases}
\end{equation*}

Below, we will mostly concentrate on the case $p=2$, \ $k=1$. The 5-cocycle resulting from the above items \ref{i:r}--\ref{i:e} will be denoted as~$c$. Its value on a 5-simplex~$w$ can be expressed as
\begin{equation}\label{h2}
 c(\nu, \eta) =
 \sum _{\substack{v,v'\subset w\\ v < v'}} Q_v(\nu, \eta) Q_{v'}(\nu, \eta)
 +\sum _{v\subset w} \tilde{\epsilon}_v^{(w)} \bigl( Q_v(\nu, \eta) \bigr)^2,
\end{equation}
where $v$ and~$v'$ are faces of~$w$. We assume in~\eqref{h2} that these faces are \emph{numbered}, and understand their numbers when writing ``\,$v < v'$\,''.

Our calculations for specific manifolds below in Section~\ref{s:c} show that \eqref{h2} is a \emph{nontrivial} cocycle---not a coboundary.

Note, by the way, that in characteristic two, a Frobenius endomorphism applied to~\eqref{cvxvyv} gives
\begin{equation*}
\sum _{v\subset w} \bigl( Q_v(\nu, \eta) \bigr)^2 = 0.
\end{equation*}
Hence, $\tilde{\epsilon}_v^{(w)}$ in~\eqref{h2} can be replaced painlessly by $1-\tilde{\epsilon}_v^{(w)}$, if needed.

\subsection{Explicit expression in characteristic three} 

For $p=3$, \ $k=1$, the value on a 5-simplex~$w$ of the 5-cocycle resulting from items \ref{i:r}--\ref{i:e} is
\begin{equation}\label{h3}
\sum _{\substack{v_1,v_2,v_3\subset w\\ v_1 < v_2 < v_3}} \epsilon _{v_1}^{(w)} \epsilon _{v_2}^{(w)} \epsilon _{v_3}^{(w)}\, Q_{v_1}(\nu, \eta)\, Q_{v_2}(\nu, \eta)\, Q_{v_3}(\nu, \eta),
\end{equation}
where $v_1$, $v_2$ and~$v_3$ are faces of~$w$.

It can be shown by a direct calculation that \eqref{h3} is again a nontrivial cocycle.

\section{Manifold invariant as a `stable bipolynomial' up to linear transforms of its variables}\label{s:sb}

\subsection{A bipolynomial function of two permitted colorings of a PL manifold with a given triangulation and a 3-cocycle}

Let $M$ be now a 5-dimensional \emph{closed} PL manifold.
Let $\upxi = \{\xi _1, \xi _2, \ldots \}$ denote an infinite (countable) set of variables over a simple field~$\mathbb F_p$, and let $F = \mathbb F_p (\upxi)$ be the field of rational functions of these variables with coefficients in~$\mathbb F_p$.

Let there be given a triangulation of~$M$ and a simplicial 3-cocycle~$\omega$ on it whose components~$\omega_{ijkl}$ belong to~$F$.
If $p \ne 2$, let also $M$ be \emph{oriented}.

For a given pair $(\rho,\sigma)$ of permitted colorings of~$M$, consider the value
\begin{equation}\label{sc}
\sum _{\mathrm{all}\;\Delta^5 \subset M} \mathfrak c_{\Delta^5} (\rho,\sigma) ,
\end{equation}
where $\mathfrak c$ is a bipolynomial 5-cocycle in finite characteristic such as \eqref{h2} for $p=2$ or~\eqref{h3} for $p=3$. Expression~\eqref{sc} is a bipolynomial function of permitted colorings $\rho$ and~$\sigma$ ---that is, of the coordinates of vectors $\rho$ and~$\sigma$ in the linear space~$V_p$ of permitted colorings w.r.t.\ some basis.

\subsection{Extracting what is invariant under Pachner moves and changing $\omega$ within its cohomology class}

\begin{definition}\label{d:I}
We define $I(M,\omega)$ as expression~\eqref{sc} considered as a function of the coordinates of vectors $\rho$ and~$\sigma$ and taken up to an $F$-linear change of these coordinates and to a `stabilization'---adding more coordinates on which \eqref{sc} does not depend, or removing such coordinates.
\end{definition}

The dimension of~$V_p$ can change under Pachner moves, so there must be no surprise that we define $I(M,\omega)$ in such a `stable' way.

\begin{theorem}\label{th:uc}
$I(M,\omega)$ is an invariant of the pair ($M$, cohomology class of\/~$\omega$).
\end{theorem}

\begin{proof}
\textbf{Invariance under Pachner moves.}
Recall that a Pachner move replaces a cluster~$C_{\mathrm{ini}}$ of 5-simplices with another cluster~$C_{\mathrm{fin}}$ in such way that these clusters form together~$\partial\Delta^6$---the boundary of a 6-simplex.

Let $\nu$ and~$\eta$ denote 3-cocycles representing $\rho$ and~$\sigma$ locally---on the simplices involved in the Pachner move.

First, we can extend $\omega$, $\nu$ and~$\eta$ onto~$C_{\mathrm{fin}}$ (glued to~$M$ by the boundary $\partial C_{\mathrm{fin}} = \partial C_{\mathrm{ini}}$). Concerning~$\omega$, we note that, once we have then removed~$C_{\mathrm{ini}}$, our new~$\omega$ obviously belongs to the same cohomology class---just because the operation was local, made within a piecewise linear 5-ball. Concerning $\nu$ and~$\eta$, we note that we have thus extended the corresponding permitted colorings $\rho$ and~$\sigma$.

Second, we are dealing with a polygon cocycle, hence
\begin{equation}\label{wD}
\pm\! \sum _{\Delta^5\subset \Delta^6}\! \mathfrak c_{\Delta^5}(\rho,\sigma)\; = \sum _{\Delta^5\subset C_{\mathrm{ini}}}\! \mathfrak c_{\Delta^5}(\rho,\sigma)\; - \! \sum _{\Delta^5\subset C_{\mathrm{fin}}}\! \mathfrak c_{\Delta^5}(\rho,\sigma) = 0.
\end{equation}
The minus sign in the middle part of~\eqref{wD} is due to the fact that the mutual orientation of the two clusters induced by an orientation of~$\Delta^6$ is \emph{opposite} to their mutual orientation in the situation when one of them replaces the other within a triangulation of~$M$. Hence, expression~\eqref{sc} remains the same under a Pachner move.

\smallskip

Third, take coordinates of vectors $\rho$ and~$\sigma$ as follows:
\begin{itemize}\itemsep 0pt
 \item for both the lhs and rhs of the Pachner move---respective maximal sets of linearly independent coordinates of $\rho$ and~$\sigma$ on the unchanged part of~$M$, that is the boundary $\partial C_{\mathrm{fin}} = \partial C_{\mathrm{ini}}$ plus the exterior of the move,
 \item plus the remaining linearly independent coordinates belonging to the inner part of the lhs or rhs, respectively.
\end{itemize}

Expression~\eqref{sc} does not depend on the latter kind of coordinates (any such dependence would make it possible to change just one term in the middle part of~\eqref{wD}), but their number may (and does) change.

\smallskip

\textbf{Invariance under adding a simplicial 3-coboundary to~$\omega$.}
It is enough to consider just adding a multiple of the coboundary of 2-cochain taking value~1 on triangle~$ijk$ and~0 on other triangles. We denote this coboundary simply as $\delta(ijk)$, and its mentioned multiple as~$\beta \delta(ijk)$. To show that $I$ does not change,
we make a few Pachner moves so as to make~$ijk$ disappear, and then the inverse Pachner moves in the inverse order,
returning thus to the initial triangulation, and
introducing the properly changed~$\omega$ on this way.
One small complication that might occur on this way is a possible violation of condition~\eqref{onv} for some tetrahedron appearing on the way. To avoid this, we take a new---not involved in any formulas before we started the inverse Pachner moves---indeterminate~$\xi _{\ell}$ from our infinite set~$\upxi$, and add $\xi _{\ell\,} \delta(ijk)$ instead of $\beta \delta(ijk)$. After all moves are done, we can assign value~$\beta$, or in fact any value, to~$\xi _{\ell}$, because $I$ does not anyhow depend on it.
\end{proof}

\subsection{Simplicial 3-coboundaries, g-colorings, and the sum~\eqref{sc}}\label{ss:ipg}

Let manifold~$M$ be as above: closed PL 5-manifold, oriented if characteristic $p\ne 2$.

\begin{definition}
 We call its permitted coloring a \emph{g-coloring} if it corresponds to a (global) simplicial 3-cocycle that is the \emph{coboundary}~$\delta\beta$ of some 2-cocycle~$\beta$.
\end{definition}

That is, the permitted colorings of any 5-simplex~$\Delta^5$ are determined by the restriction of 3-cocycle~$\delta \beta$ on~$\Delta^5$. We use the term ``g-coloring'' to be in conformity with~\cite{odd-gons}.

\begin{proposition}\label{p:pf}
Expression~\eqref{sc} does not change if we add a g-coloring to either~$\rho$ or~$\sigma$, or both (for a given triangulation of~$M$).
\end{proposition}

\begin{proof}
Enough to consider the case where we change~$\rho$ by the coboundary of the 2-cochain taking value~$1$ on one triangle~$s$ and $0$ on all other triangles. Then, we make a series of Pachner moves leading to the disappearance of~$s$. After that, we can introduce the new version of~$\rho$ while making the inverse moves in the inverse order.
\end{proof}

\subsection{Factoring modulo g-colorings}
Denote $V_p$ the linear space of permitted colorings of~$M$, and $V_g$ ---the space of g-colorings. In view of Proposition~\ref{p:pf}, it makes sense to consider $I(M,\omega)$ as a bipolynomial form on~$V_p/V_g$. Moreover, stabilization is no longer needed, as the following proposition shows.

\begin{proposition}\label{p:pg}
The dimension of linear space~$V_p/V_g$ is an invariant of the pair~$(M,\omega)$.
\end{proposition}

\begin{proof}
Any Pachner move makes, as we remember, only a local change of triangulation---replacement of cluster~$C_{\mathrm{ini}}$ by~$C_{\mathrm{fin}}$, with the same boundary. For any of these clusters, the linear spaces of restrictions of g-colorings and general permitted colorings are the same and correspond, locally, to 3-cocycles modulo~$\omega$. Hence, the changes of $\dim V_p$ and~$\dim V_g$ are the same under any Pachner move.
\end{proof}

\section{Calculations}\label{s:c}

\subsection{Calculations in finite fields}

It turns out that, for actual computer calculations, a large enough \emph{finite} field can be taken instead of our $F=\mathbb F_2(\upxi)$. At least, it applies to the examples given below in this section, where 3-cocycles~$\omega$ belong to cohomology classes with coefficients in the \emph{simple field}~$\mathbb F_2$.

That is, we calculate cocycles representing a 3-cohomology basis of~$M$ with coefficients in~$\mathbb F_2$, but then a coboundary with coefficients in a larger field must be added to it to ensure condition~\eqref{onv}.

We provide some results of computer calculations for specific manifolds, and for the 5-cocycle $\mathfrak c = c$~\eqref{h2} in characteristic two, using finite fields like~$F_{\mathrm{fin}} = \mathbb F_{2^{15}}$. 
Remarkably, in all examples below, our bipolynomial invariant turns out to be a symmetric \emph{bilinear} form of the \emph{squares} of its variables (the same phenomenon was observed for simpler invariants in~\cite{odd-gons})!

Hence, matrix~$A$ of this bilinear form appears whose \emph{rank} is a manifold invariant, because it is determined by~$I(M,\omega)$, and does not change under all transformations mentioned in Definition~\ref{d:I}.

\begin{proposition}
 The calculation results, namely, $\dim ( V_p / V_g )$ and $\rank A$, do not change after replacing a finite field $F_{\mathrm{fin}}$ with $F=\mathbb F_2(\upxi)$.
\end{proposition}

\begin{proof}
The idea is in replacing elements of~$F_{\mathrm{fin}}$ by indeterminates over~$\mathbb F_2$. The cohomology class is not affected because it belongs to~$\mathbb F_2$.

Let  vector space $V$ over $F_{\mathrm{fin}}$ be either $V_p$ or $V_g$. First, we pass from $V$ to
 \begin{equation*}
  W = F_{\mathrm{fin}} \otimes _{\mathbb F_2} \mathbb F_2 ( \upxi )
 \end{equation*}
 and consider $W$ as a vector space over $F_{\mathrm{fin}} (\upxi) = \mathbb F_2(\upxi) \otimes _{\mathbb F_2} F_{\mathrm{fin}}$.
 
 Doing necessary Pachner moves and changing components of~$\omega$, we can make them belong to $\mathbb F_2(\upxi) \subset F_{\mathrm{fin}}(\upxi)$ (as $\omega$ belongs to a cohomology class with coefficients in~$\mathbb F_2$, no coefficients from~$F_{\mathrm{fin}}$ other than $0$ and~$1$ are needed). This way, $W$ acquires the form
 \begin{equation*}
  \tilde W = \tilde V \otimes _{\mathbb F_2} F_{\mathrm{fin}},
 \end{equation*}
 where $\tilde V$ is a vector space over $\mathbb F_2 (\upxi)$.
 
 Now the tensor multiplication by $F_{\mathrm{fin}}$ can be removed everywhere, with passing from field $F_{\mathrm{fin}}(\upxi)$ to $F = \mathbb F_2 (\upxi)$. We see that the results for field $F_{\mathrm{fin}}$ are the same as for~$F$.
\end{proof}

\subsection{Notations and remarks}

Below, $[\omega ]$ means the cohomology class of the simplicial 3-cocycle~$\omega$ over~$\mathbb F_2$; we also take the liberty to call~$[\omega ]$ simply ``cocycle''. The whole number of cocycles (in this sense) is~$2^{\dim H^3(M,\mathbb Z_2)}$.

 It follows from the examples below that, typically, for $[\omega]=0$, \emph{not all permitted colorings correspond to any (global) simplicial 3-cocycle~$\nu$}, due to the fact that $\dim V_p/V_g > \dim H^3(M,\mathbb Z_2)$. Moreover, similar phenomenon can be seen to take place even for some nonzero~$[\omega]$'s!

Letter~$K$ means below the Klein bottle.

\subsection{$M = \mathbb RP^2 \times S^3$}\label{ss:rp2s3}

The third cohomology group is one-dimen\-sional: $\dim H^3(M,\mathbb Z_2)=1$, denote its generator~$c$. There are hence just two cocycles to consider, as follows.

$[\omega ] = 0$: \quad $\dim V_p/V_g = 2$, \ $\rank A = 2$.

$[\omega ] = c$: \quad $\dim V_p/V_g = 0$, \ $\rank A = 0$.

\subsection{$M = \mathbb RP^4 \times S^1$}\label{ss:rp4s1}

Denote $a$ a generator of the first cohomology group of~$\mathbb RP^4$, and $b$ a generator of the first cohomology group of~$S^1$.

$[\omega ] = 0$: \quad $\dim V_p/V_g = 4$, \ $\rank A = 4$.

$[\omega ] = a^2 \smile b$: \quad $\dim V_p/V_g = 2$, \ $\rank A = 2$.

$[\omega ] = a^3$: \quad $\dim V_p/V_g = 1$, \ $\rank A = 0$.

$[\omega ] = a^3 + a^2 \smile b$: \quad $\dim V_p/V_g = 1$, \ $\rank A = 0$.

\subsection{$M = K \times \mathbb RP^3$}\label{ss:kbrp3}

$[\omega ] = 0$: \quad $\dim V_p/V_g = 7$, \ $\rank A = 4$.

Other cocycles:

\begin{center}
\begin{tabular}{|c|c|c|}
\hline
$\dim V_p/V_g$ & $\rank A$ & number of such cocycles \\ \hline \hline
3 & 0 & 4 \\ \hline
4 & 0 & 2 \\ \hline
4 & 2 & 8 \\ \hline
5 & 2 & 1 \\ \hline
\end{tabular}
\end{center}

\subsection{$M = \mathbb RP^2 \times \mathbb RP^3$}\label{ss:rp2rp3}

$[\omega ] = 0$: \quad $\dim V_p/V_g = 5$, \ $\rank A = 4$.

Other cocycles:

\begin{center}
\begin{tabular}{|c|c|c|}
\hline
$\dim V_p/V_g$ & $\rank A$ & number of such cocycles \\ \hline \hline
2 & 0 & 4 \\ \hline
3 & 2 & 3 \\ \hline
\end{tabular}
\end{center}

\subsection{$M = S^1 \times K \times \mathbb RP^2$}\label{ss:s1kbrp2}


$[\omega ] = 0$: \quad $\dim V_p/V_g = 11$, \ $\rank A = 8$.

Other cocycles:

\begin{center}
\begin{tabular}{|c|c|c|}
\hline
$\dim V_p/V_g$ & $\rank A$ & number of such cocycles \\ \hline \hline
6 & 0 & 56 \\ \hline
7 & 2 & 36 \\ \hline
8 & 4 & 34 \\ \hline
9 & 6 & 1 \\ \hline
\end{tabular}
\end{center}

\subsection{$M = S^1 \times \mathbb RP^2 \times \mathbb RP^2$}\label{ss:s1rp2rp2}


$[\omega ] = 0$: \quad $\dim V_p/V_g = 8$, \ $\rank A = 6$.

Other cocycles:

\begin{center}
\begin{tabular}{|c|c|c|}
\hline
$\dim V_p/V_g$ & $\rank A$ & number of such cocycles \\ \hline \hline
4 & 0 & 12 \\ \hline
5 & 2 & 16 \\ \hline
6 & 4 & 3 \\ \hline
\end{tabular}
\end{center}

\subsection{$M = S^2 \times \mathbb RP^3$}\label{ss:s2rp3}


$[\omega ] = 0$: \quad $\dim V_p/V_g = 3$, \ $\rank A = 0$.

Other cocycles:

\begin{center}
\begin{tabular}{|c|c|c|}
\hline
$\dim V_p/V_g$ & $\rank A$ & number of such cocycles \\ \hline \hline
1 & 0 & 2 \\ \hline
2 & 0 & 1 \\ \hline
\end{tabular}
\end{center}

\subsection{$M = S^2 \times S^3$}\label{ss:s2s3}


$[\omega ] = 0$: \quad $\dim V_p/V_g = 1$, \ $\rank A = 0$.

There is of course the only nonzero 3-cocycle, and for it: \ $\dim V_p/V_g = 0$, \ $\rank A = 0$.

\subsection{$M = K \times S^3$}\label{ss:kbs3}


$[\omega ] = 0$: \quad $\dim V_p/V_g = 3$, \ $\rank A = 2$.

There is the only nonzero 3-cocycle, and for it: \ $\dim V_p/V_g = 0$, \ $\rank A = 0$.


\begin{thebibliography}{99}

\bibitem{pcst}
Ilaria Colazzo, Eric Jespers, and Lukasz Kubat,
\textit{Set-theoretic solutions of the pentagon equation},
\href{https://doi.org/10.1007/s00220-020-03862-6}{Commun. Math. Phys. {\bf 380} (2020), 1003–1024}.

\bibitem{DM-H}
Aristophanes Dimakis, Folkert M\"uller-Hoissen,
\textit{Simplex and Polygon Equations},
\href{https://doi.org/10.3842/SIGMA.2015.042}{SIGMA {\bf 11} (2015), paper~042, 49~pages}.
\href{https://arxiv.org/abs/1409.7855}{arXiv:1409.7855}.

\bibitem{discr-val}
Ivan B. Fesenko and Sergei V. Vostokov,
\textit{Local fields and their extensions},
Translations of Mathematical Monographs, Volume~121 (Second ed.),
Providence, RI: American Mathematical Society, 2002.

\bibitem{YB-from-pentagon}
R. M. Kashaev, 
\textit{The Heisenberg double and the pentagon relation},
Algebra i Analiz, {\bf 8}:4 (1996), 63--74; 
St. Petersburg Math. J., {\bf 8}:4 (1997), 585--592.

\bibitem{cubic}
I.G.~Korepanov, N.M.~Sadykov,
\textit{Hexagon cohomologies and polynomial TQFT actions},
\href{https://arxiv.org/abs/1707.02847}{arXiv:1707.02847}.

\bibitem{KS}
I.G.~Korepanov, N.M.~Sadykov,
\textit{Pentagon relations in direct sums and Grassmann algebras},
\href{https://doi.org/10.3842/SIGMA.2013.030}{SIGMA {\bf 9} (2013), 030, 16 pages}.

\bibitem{nonconstant-hexagon}
I.G.~Korepanov,
\textit{Nonconstant hexagon relations and their cohomology},
\href{https://doi.org/10.1007/s11005-020-01338-1}{Lett. Math. Phys. {\bf 111}, 1 (2021)}.
\href{https://arxiv.org/abs/1812.10072}{arXiv.org:1812.10072}.

\bibitem{odd-gons}
Igor G. Korepanov,
\textit{Odd-gon relations and their cohomology},
\href{https://doi.org/10.1016/j.padiff.2024.100856}{Partial Differ. Equ. Appl. Math.
{\bf 11} (2024), 100856}

\bibitem{Lickorish}
W.B.R. Lickorish,
\textit{Simplicial moves on complexes and manifolds},
Geom. Topol. Monogr. {\bf 2} (1999), 299--320.
\href{http://www.arxiv.org/abs/math/9911256}{arXiv:math/9911256}.

\bibitem{Macdonald}
I.G. Macdonald,
\textit{Symmetric functions and Hall polynomials},
Second Edition,
Oxford University Press, 1995.

\bibitem{Pachner}
U. Pachner,
\textit{PL homeomorphic manifolds are equivalent by elementary shellings},
Europ. J. Combinatorics {\bf 12} (1991), 129--145.

\end{thebibliography}
\end{document}